\documentclass[a4paper,11pt]{amsart}

\usepackage{amsmath, amsthm, amsfonts, amssymb}
\usepackage[bookmarks=false]{hyperref}
\usepackage{enumerate}

\usepackage{color} 

\numberwithin{equation}{section}

\newtheorem{thm}{Theorem}[]
\newtheorem{lem}[thm]{Lemma}

\newtheorem{cor}[thm]{Corollary}

\theoremstyle{definition}

\newtheorem{prob}[thm]{Problem}

\newcommand{\R}{\mathbf{R}}
\newcommand{\C}{\mathbf{C}}
\newcommand{\Z}{\mathbf{Z}}

\newcommand{\N}{\mathbf{N}}

\newcommand{\Ad}{\operatorname{Ad}}

\newcommand{\rL}{\mathord{\text{\rm L}}}

\newcommand{\rB}{\mathord{\text{\rm B}}}
\newcommand{\rC}{\mathord{\text{\rm C}}}

\newcommand{\rE}{\mathord{\text{\rm E}}}

\newcommand{\Ball}{\mathord{\text{\rm Ball}}}
\newcommand{\spn}{\mathord{\text{\rm span}}}

\newcommand{\AC}{\mathord{\text{\rm AC}}}

\newcommand{\dpr}{^{\prime\prime}}

\begin{document}

\title[Singular MASAs in type III factors and Connes' Bicentralizer Property]{Singular MASAs in type III factors and \\ Connes' Bicentralizer Property}

\begin{abstract}
We show that any type ${\rm III_1}$ factor with separable predual satisfying Connes' Bicentralizer Property (CBP) has a singular maximal abelian $\ast$-subalgebra that is the range of a normal conditional expectation. We also investigate stability properties of CBP under finite index extensions/restrictions of type ${\rm III_1}$ factors.
\end{abstract}

\author{Cyril Houdayer}
\address{Laboratoire de Math\'ematiques d'Orsay\\ Universit\'e Paris-Sud\\ CNRS\\ Universit\'e Paris-Saclay\\ 91405 Orsay\\ FRANCE}
\email{cyril.houdayer@math.u-psud.fr}
\thanks{CH is supported by ERC Starting Grant GAN 637601}

\author{Sorin Popa}
\address{Mathematics Department \\ University of California at Los Angeles \\ CA 90095-1555 \\ USA}
\email{popa@math.ucla.edu}
\thanks{SP is supported by NSF Grant DMS-1400208, a Simons Fellowship and Chaire d'Excellence de la FSMP~2016}

\subjclass[2010]{46L10, 46L36}
\keywords{Connes' bicentralizer property; Singular maximal abelian $\ast$-subalgebras; Type ${\rm III}$ factors}

\maketitle


\section{Introduction}

Let $M$ be any von Neumann algebra and $A \subset M$ any maximal abelian $\ast$-subalgebra (abbreviated MASA). Denote by $\mathcal N_M(A) = \{u \in \mathcal U(M) : u A u^* = A\}$ the group of unitaries in $M$ that normalize $A$. We say that $A \subset M$ is {\em singular} if $\mathcal N_M(A) = \mathcal U(A)$, that is, the only unitaries in $M$ that normalize $A$ are in $A$. 
It has been shown in \cite{Po82} that any diffuse semifinite von Neumann algebra with separable predual and any type ${\rm III_\lambda}$ factor ($0 \leq \lambda < 1$) with separable predual has a singular MASA.  A new approach to 
this result has been recently given in  \cite{Po16}.  But while many examples of  type ${\rm III_1}$ factors are known to have singular MASAs,  the problem of whether any type ${\rm III_1}$ factor has a singular MASA remained open.

Following \cite{Co80}, 
if $M$ is a type ${\rm III_1}$ factor with a normal faithful state 
$\varphi$,  then the {\em bicentralizer} of $(M, \varphi)$ is defined by
$$\rB(M, \varphi) = \left \{ a \in M : \lim_n \|x_n a - a x_n\|_\varphi = 0, \forall (x_n)_n \in \AC(M, \varphi) \right \}$$
where $\AC(M, \varphi) = \left\{ (x_n)_n \in \ell^\infty(\N, M) : \lim_n \|x_n \varphi - \varphi x_n\| = 0\right \}$. It is known that $\rB(M, \varphi) \subset M$ is a von Neumann subalgebra that is globally invariant under the modular flow $\sigma^\varphi$. By Connes--St\o rmer transitivity theorem (\cite{CS76}), it follows that if $\rB(M, \varphi) = \C 1$ for some normal faithful state $\varphi$ 
on $M$, then $\rB(M, \psi) = \C 1$ for any normal faithful state $\psi$ on $M$ (cf. \cite[Corollary 1.5]{Ha85}). 
We say that $M$ satisfies Connes' Bicentralizer Property (abbreviated CBP) if $\rB(M, \varphi) = \C 1$ for some 
(equivalently, for any) normal faithful state $\varphi$ on $M$. 

Haagerup showed in \cite{Ha85} that any amenable type ${\rm III_1}$ factor with separable predual satisfies CBP. 
Together with the work of Connes \cite{Co85}, this 
showed the uniqueness of the amenable factor of type ${\rm III_1}$ with separable predual. Haagerup also obtained in \cite[Theorem 3.1]{Ha85} the following characterization of CBP: a type ${\rm III_1}$ factor $M$ with separable predual satisfies CBP if and only if it has a 
normal faithful state $\varphi$ such that $(M_\varphi)' \cap M = \C 1$. 
Several classes of nonamenable type ${\rm III_1}$ factors have been shown  to satisfy CBP: free Araki--Woods factors \cite{Ho08}; free product factors \cite{HU15}; nonamenable factors satisfying Ozawa's condition (AO) \cite{HI15}. However, Connes' Bicentralizer Problem 
is still open for arbitrary type ${\rm III_1}$ factors.

In this note, we prove that every factor $M$ with separable predual that has a normal faithful state $\varphi$ 
satisfying the condition $(M_\varphi)' \cap M = \C 1$, contains an abelian $\ast$-subalgebra $A \subset M_\varphi$ that's maximal abelian and singular in $M$ (see Theorem \ref{thm-singular}). We prove this result 
by adapting the argument used in \cite[Theorem 2.1]{Po16} for the type ${\rm II_1}$ case. By combining our result with Haagerup's characterization of CBP \cite{Ha85} explained above, we derive that any type ${\rm III_1}$ factor with separable predual satisfying CBP has a singular MASA that is the range of a normal conditional expectation. 
We end the paper with some results and comments about the stability of Connes' Bicentralizer Property 
for inclusions of type ${\rm III_1}$ factors with normal conditional expectation, under the assumption that the inclusion has finite index (see Theorem \ref{thm-CBP-finite-index}).

\subsection*{Acknowledgments} The first named author is grateful to Yusuke Isono for useful discussions.

\section*{Notation}

Let $M$ be any $\sigma$-finite von Neumann algebra. We denote by $(M, \rL^2(M), J, \rL^2(M)_+)$ the standard form of $M$. We moreover denote by $\mathcal Z(M)$ its center, by $\mathcal U(M)$ its group of unitaries and by $\Ball(M)$ its unit ball with respect to the uniform norm $\|\cdot\|_\infty$. For any normal faithful state $\varphi$ on $M$, we denote by $\xi_\varphi \in \rL^2(M)_+$ its canonical implementing vector, by $\sigma^\varphi$ its modular automorphism group and by $M_\varphi = \{x \in M : \sigma_t^\varphi(x) = x, \forall t \in \R\}$ its centralizer. We write $\|x\|_\varphi = \|x \xi_\varphi\|$ for every $x \in M$.

We say that a von Neumann subalgebra $N \subset M$ is with {\em normal conditional expectation} (abbreviated NCE) if there exists a normal faithful conditional expectation $\rE_N : M \to N$. Recall that $N \subset M$ is globally invariant under the modular automorphism group $\sigma^\varphi$ if and only if there exists a $\varphi$-preserving conditional expectation $\rE_N^\varphi : M \to N$ (see \cite[Theorem IX.4.2]{Ta03}).

\section{Singular MASAs in type ${\rm III_1}$ factors}\label{section:singular-MASA}

We recall from \cite{Po81} two results that will be used in the proof of Theorem \ref{thm-singular}. 

\begin{lem}[{\cite[Theorem 2.5]{Po81}}]\label{lem-technical-1}
Let $M$ be any $\sigma$-finite von Neumann algebra, $\varphi$ any normal faithful state on $M$ and $N \subset M_\varphi$ be any von Neumann subalgebra such that $N' \cap M \subset N$.

For any finite dimensional abelian $\ast$-subalgebra $D \subset N$, any $x_1, \dots, x_n \in M$ and any $\varepsilon > 0$, there exists a finite dimensional abelian $\ast$-subalgebra $A \subset N$ that contains $D$ and for which we have
\begin{equation*}
\forall 1 \leq i \leq n, \quad \|\rE_{A' \cap M}^\varphi(x_i) - \rE_A^\varphi(x_i)\|_\varphi \leq \varepsilon.
\end{equation*}
\end{lem}

\begin{proof}
Write $D = \bigoplus_{j \in J} \C e_j$ where $J$ is a nonempty finite set and $(e_j)_{j \in J}$ are the nonzero minimal projections of $D$. For every $j \in J$, we have $e_j N e_j \subset (e_j M e_j)_{\varphi_{e_j}}$ and $(e_j N e_j)' \cap e_j M e_j \subset e_j N e_j$ (see \cite[Lemma 2.1]{Po81}). 

Let $x_1, \dots, x_n \in M$ and $\varepsilon > 0$. For every $j \in J$, by \cite[Theorem 2.5]{Po81}, there exists a finite dimensional abelian $\ast$-subalgebra $A_j \subset e_j N e_j$ such that 
\begin{equation*}
\forall 1 \leq i \leq n, \quad \|\rE^{\varphi_{e_j}}_{A_j' \cap e_jMe_j}(e_j x_i e_j) - \rE^{\varphi_{e_j}}_{A_j}(e_jx_ie_j)\|_{\varphi_{e_j}} \leq \varepsilon.
\end{equation*}
Put $A = \bigoplus_{j \in J} A_j$. Then $A \subset N$ is a finite dimensional abelian $\ast$-subalgebra that contains $D$. Moreover, for all $1 \leq i \leq n$, we have
\begin{align*}
\|\rE_{A' \cap M}^\varphi(x_i) - \rE_A^\varphi(x_i)\|_\varphi^2 &= \sum_{j \in J} \varphi(e_j) \|\rE^{\varphi_{e_j}}_{A_j' \cap e_jMe_j}(e_j x_i e_j) - \rE^{\varphi_{e_j}}_{A_j}(e_jx_ie_j)\|_{\varphi_{e_j}}^2 \\
&\leq \sum_{j \in J} \varphi(e_j) \varepsilon^2 = \varepsilon^2. \qedhere
\end{align*}
\end{proof}

\begin{lem}[{\cite[Theorem 3.2]{Po81}}]\label{lem-technical-2}
Let $M$ be any factor with separable predual, $\varphi$ any normal faithful state on $M$ and $N \subset M_\varphi$ be any von Neumann subalgebra such that $N' \cap M = \C 1$.

For any finite dimensional abelian $\ast$-subalgebra $D \subset N$, there exists an abelian $\ast$-subalgebra $A \subset N$ that contains $D$ and that is maximal abelian in $M$.
\end{lem}

\begin{proof}
Write $D = \bigoplus_{j \in J} \C e_j$ where $J$ is a nonempty finite set and $(e_j)_{j \in J}$ are the nonzero minimal projections of $D$. For every $j \in J$, we have $e_j N e_j \subset (e_j M e_j)_{\varphi_{e_j}}$ and $(e_j N e_j)' \cap e_j M e_j = \C e_j$ (see \cite[Lemma 2.1]{Po81}). For every $j \in J$, by \cite[Theorem 3.2]{Po81}, there exists an abelian $\ast$-subalgebra $A_j \subset e_j N e_j$ that is maximal abelian in $e_j M e_j$. Put $A = \bigoplus_{j \in J} A_j$.  Then $A \subset N$ is an abelian $\ast$-subalgebra that contains $D$ and that is maximal abelian in $M$.
\end{proof}

\begin{thm}\label{thm-singular}
Let $M$ be any non-type ${\rm I}$ factor with separable predual, $\varphi$ any normal faithful state on $M$ and $N \subset M_\varphi$ any subalgebra such that $N' \cap M = \C 1$.

Then there exists an abelian $\ast$-subalgebra $A \subset N$ that is maximal abelian and singular in $M$.
\end{thm}

\begin{proof}
We follow the lines of the proof of \cite[Theorem 2.1]{Po16}. Choose a sequence $x_n \in \Ball(M)$ that is $\ast$-strongly dense in $\Ball(M)$ and a sequence of projections $e_n \in N$ that is strongly dense in the set of all projections of $N$ with $e_0 = 1$. We may further assume that each projection $e_n$ appears infinitely many times in the sequence $(e_m)_{m \in \N}$.

We construct inductively an increasing sequence $A_n$ of finite dimensional abelian $\ast$-subalgebras of $N$ together with a sequence of projections $f_n \in A_n$ and a sequence of unitaries $v_n \in \mathcal U(A_n f_n)$ satisfying the following properties:

\begin{itemize}

\item [(P1)] $\|f_n - e_n\|_\varphi \leq 7 \|e_n - \rE^\varphi_{A_{n - 1}' \cap N}(e_n)\|_\varphi$;

\item [(P2)] $\|\rE^\varphi_{A_n' \cap M}(x_i^*v_n x_j)f_n^\perp\|_\varphi \leq 2^{-n}$ for all $0 \leq i, j \leq n$;

\item [(P3)] $\|\rE^\varphi_{A_n' \cap M}(x_j) - \rE^\varphi_{A_n}(x_j)\|_\varphi \leq 2^{-n}$ for all $0 \leq j \leq n$.

\end{itemize}

We put $A_{-1} = A_0 = \C 1$, $f_0 =  v_0 = 1$. Assume that we have constructed $(A_k, f_k, v_k)$ for all $0 \leq k \leq n$. Put $f_{n + 1} := \mathbf 1_{[1/2, 1]}(\rE^\varphi_{A_n'\cap N}(e_{n + 1}))$. Then $f_{n + 1} \in A_n' \cap N$ is a projection that satisfies $\|f_{n + 1} - e_{n + 1}\|_\varphi \leq 7 \|e_{n + 1} - \rE^\varphi_{A_n' \cap N}(e_{n + 1})\|_\varphi$ by \cite[Lemma 1.4]{Po82}. Then (P1) holds true for $f_{n + 1}$. 

Assume that $f_{n + 1} \in \{ 0, 1 \}$. Then (P2) holds true for any choice of $A_{n + 1}$. By Lemma \ref{lem-technical-1}, we may find a finite dimensional abelian $\ast$-subalgebra $A_{n + 1} \subset N$ that contains $A_n$ and that satisfies 
\begin{equation}\label{eq-MASA}
\forall 0 \leq i \leq n + 1, \quad \|\rE^\varphi_{A_{n + 1}'\cap M }(x_i) - \rE^\varphi_{A_{n + 1}}(x_i) \|_\varphi \leq 2^{-(n + 1)}.
\end{equation}
Thus, \eqref{eq-MASA} shows that (P3) holds true for $A_{n + 1}$.

Assume that $f_{n + 1} \not\in \{ 0, 1 \}$. By Lemma \ref{lem-technical-2}, there exists an abelian $\ast$-subalgebra $B \subset  N$ that contains $A_n \vee \C f_{n + 1}$ and that is maximal abelian in $M$. Since $(A_n f_{n + 1})' \cap  f_{n + 1} N f_{n + 1}$ is a type ${\rm II_1}$ von Neumann algebra and $B f_{n + 1}^\perp \subset  f_{n + 1}^\perp M f_{n + 1}^\perp$ is an abelian subalgebra with normal expectation, \cite[Theorem 2.3]{HV12} (see also \cite[Theorem 2.1 and Corollary 2.3]{Po03}) implies that there exists $v_{n + 1} \in \mathcal U((A_n f_{n + 1})' \cap  f_{n + 1} N f_{n + 1})$ for which  we have
\begin{equation}\label{eq-intertwining-1}
\forall 0 \leq i, j \leq n + 1, \quad \left\|\rE^{\varphi_{f_{n + 1}^\perp}}_{B f_{n + 1}^\perp}(f_{n + 1}^\perp  x_i^*f_{n + 1} \, v_{n + 1} \, f_{n + 1}  x_j   f_{n + 1}^\perp) \right\|_\varphi < 2^{-(n + 2)}.
\end{equation}
Using the spectral theorem, we may further assume that $v_{n + 1}  \in \mathcal U((A_n f_{n + 1})' \cap  f_{n + 1} N f_{n + 1})$ has finite spectrum and still satisfies \eqref{eq-intertwining-1}. We may then choose a finite dimensional abelian $\ast$-subalgebra $D_1 \subset Bf_{n + 1}$ that contains $A_n f_{n + 1}$ and $v_{n + 1}$. Moreover, using \cite[Lemma 1.2]{Po81} and \eqref{eq-intertwining-1}, we may choose a finite dimensional abelian $\ast$-subalgebra $D_2 \subset Bf_{n + 1}^\perp$ that contains $A_n f_{n + 1}^\perp$ and for which we have
\begin{equation}\label{eq-intertwining-2}
\forall 0 \leq i, j \leq n + 1, \quad \left\|\rE^{\varphi_{f_{n + 1}^\perp}}_{D_2' \cap f_{n + 1}^\perp M f_{n + 1}^\perp}(f_{n + 1}^\perp  x_i^*f_{n + 1} \, v_{n + 1} \, f_{n + 1} x_j   f_{n + 1}^\perp) \right\|_{\varphi} < 2^{-(n + 1)}.
\end{equation}
Letting $D := D_1 \oplus D_2$, we can then rewrite \eqref{eq-intertwining-2} as
\begin{equation}\label{eq-intertwining-3}
\forall 0 \leq i, j \leq n + 1, \quad \|\rE^\varphi_{D' \cap M}(f_{n + 1}^\perp  x_i^*f_{n + 1} \, v_{n + 1} \, f_{n + 1} x_j   f_{n + 1}^\perp)\|_\varphi < 2^{-(n + 1)}.
\end{equation}
By Lemma \ref{lem-technical-1}, we may find a finite dimensional abelian $\ast$-subalgebra $A_{n + 1} \subset N$ that contains $D$ (and hence that contains $A_n$) and that satisfies \eqref{eq-MASA}. Thus, \eqref{eq-MASA} shows that (P3) holds true for $A_{n + 1}$. Since $D \subset A_{n + 1}$, we have $A_{n + 1}' \cap M \subset D' \cap M$ and \eqref{eq-intertwining-3} implies that (P2) holds true for $f_{n + 1}$ and $A_{n + 1}$. Thus, we have constructed $(A_{n + 1}, f_{n + 1}, v_{n + 1})$.

Put $A = \bigvee_{n \in \N} A_n$. Property (P3) and \cite[Lemma 1.2]{Po81} imply that $A' \cap M = A$ and hence $A$ is maximal abelian in $M$. It remains to prove that $A$ is singular in $M$. By contradiction, assume that $A \neq \mathcal N_M(A)\dpr$. Choose $u \in \mathcal N_M(A) \setminus \mathcal U(A)$. We can then find a nonzero projection $z \in A$ such that $uzu^* \perp z$. Denote by $h$ the unique nonsingular (possibly) unbounded positive selfadjoint operator affiliated with $A$ such that $(\varphi \circ \Ad(u))|_A = \varphi(h \, \cdot \,) |_A$. By \cite[Lemme 1.4.5(2)]{Co72}, we have $\sigma_t^\varphi (u) = u h^{{\rm i}t} $ for every $t \in \R$. Since $h$ is nonsingular, there exists a projection $p \in A$ large enough so that $pz \neq 0$ and $\delta > 0$ so that $\delta p \leq hp \leq \delta^{-1} p$. It follows that the nonzero partial isometry $up \in M$ (resp.\ $(up)^* \in M$) is entire analytic with respect to the modular automorphism group $\sigma^\varphi$. Hence, there exists  $\kappa_1 \geq 1$ (resp.\ $\kappa_2 \geq 1$) such that $\|x \, up\|_\varphi \leq \kappa_1 \|x\|_\varphi$ (resp.\ $\|x \, (up)^*\|_\varphi \leq \kappa_2 \|x\|_\varphi$) for every $x \in M$.

Put $q := pz \in A$. For every $n \in \N$, we have
\begin{align}\label{eq-inequality-1}
\|q f_n\|_\varphi &= \|z v_n p\|_\varphi \\ \nonumber
&= \| u^* \, uz v_n pu^* \, up\|_\varphi \\ \nonumber
& \leq \kappa_1 \| uz v_n p u^* \|_\varphi \\ \nonumber
&= \kappa_1 \| \rE^\varphi_{A_n' \cap M}(uz v_npu^*) z^\perp \|_\varphi \quad (\text{since } uz v_npu^* \in z^\perp(A_n'\cap M)z^\perp).
\end{align}
Since $z, v_n \in A \subset M_\varphi$ and since $(up)^*$ is entire analytic with respect to the modular automorphism group $\sigma^\varphi$, for all $0 \leq i, j \leq n$, we further have
\begin{align}\label{eq-inequality-2}
\| \rE^\varphi_{A_n' \cap M}(uz \, v_n \, pu^*) z^\perp \|_\varphi 
&\leq \| \rE^\varphi_{A_n' \cap M}(x_i^* \, v_n \, pu^*) z^\perp \|_\varphi + \kappa_2 \|x_i^* - uz\|_\varphi  \\ \nonumber
&\leq \| \rE^\varphi_{A_n' \cap M}(x_i^* \, v_n \, x_j) z^\perp \|_\varphi + \|x_j - pu^*\|_\varphi + \kappa_2 \|x_i^* - uz\|_\varphi \\ \nonumber
&\leq \| \rE^\varphi_{A_n' \cap M}(x_i^* \, v_n \, x_j) e_n^\perp \|_\varphi + \| e_n - z\|_\varphi + \|x_j - pu^*\|_\varphi + \kappa_2 \|x_i^* - uz\|_\varphi.
\end{align}
Since $z \in A \subset A_{n - 1}'\cap N$, for every $n \in \N$, using (P1) we have
\begin{align*}
\|e_n - f_n\|_\varphi &\leq 7 \|e_n - \rE^\varphi_{A_{n - 1}' \cap N}(e_n)\|_\varphi \\
&=  7 \|e_n - z - \rE^\varphi_{A_{n - 1}' \cap N}(e_n - z)\|_\varphi \\
& \leq 7 \|e_n - z\|_\varphi.
\end{align*}
Choose now $n_0 \in \N$ large enough so that $2^{-n_0} \leq \|q\|_\varphi/(100 \kappa_1)$. By density, we may then choose $i, j \geq 0$ and $n \geq \max(i, j , n_0)$ so that $\| e_n - z\|_\varphi + \|x_j - pu^*\|_\varphi + \kappa_2 \|x_i^* - uz\|_\varphi \leq \|q\|_\varphi/(100 \kappa_1)$. Using \eqref{eq-inequality-2} and (P2), we then have
\begin{align}\label{eq-inequality-3}
\| \rE^\varphi_{A_n' \cap M}(uz v_npu^*) z^\perp \|_\varphi &\leq \| \rE^\varphi_{A_n' \cap M}(x_i^* \, v_n \, x_j) e_n^\perp \|_\varphi + \|q\|_\varphi/(100 \kappa_1) \\ \nonumber
&\leq \| \rE^\varphi_{A_n' \cap M}(x_i^* \, v_n \, x_j) f_n^\perp \|_\varphi + \|f_n - e_n\|_\varphi + \|q\|_\varphi/(100 \kappa_1) \\ \nonumber
&\leq \| \rE^\varphi_{A_n' \cap M}(x_i^* \, v_n \, x_j) f_n^\perp \|_\varphi + 7\| e_n - z\|_\varphi + \|q\|_\varphi/(100 \kappa_1) \\ \nonumber
&\leq 2^{-n}+ 7 \|q\|_\varphi/(100 \kappa_1)  + \|q\|_\varphi/(100 \kappa_1) \\ \nonumber
&\leq 9 \|q\|_\varphi/(100\kappa_1) \\ \nonumber
& \leq  \|q\|_\varphi/(2 \kappa_1). \nonumber
\end{align}
Combining \eqref{eq-inequality-1} and \eqref{eq-inequality-3}, we obtain
\begin{equation}\label{eq-conclusion-1}
\|q f_n\|_\varphi \leq \|q\|_\varphi/2.
\end{equation}
On the other hand, we have
\begin{align}\label{eq-conclusion-2}
\|q f_n - q\|_\varphi = \|q(f_n - z)\|_\varphi 
&\leq \|f_n - z\|_\varphi \\ \nonumber
&\leq \|f_n - e_n\|_\varphi + \|e_n - z\|_\varphi \\ \nonumber
& \leq 8 \|e_n - z\|_\varphi \\ \nonumber
& \leq 8 \|q\|_\varphi/(100\kappa_1) \\ \nonumber
&\leq \|q\|_\varphi/4. \nonumber
\end{align}
Combining \eqref{eq-conclusion-1} and \eqref{eq-conclusion-2}, we finally obtain
\begin{equation*}
\|q\|_\varphi/2 \geq \|q f_n\|_\varphi \geq \|q\|_\varphi - \|qf_n - q\|_\varphi \geq 3 \|q\|_\varphi/4.
\end{equation*}
Since $\|q\|_\varphi \neq 0$, this is contradiction. Therefore $A = \mathcal N_M(A)\dpr$ and hence $A \subset M$ is singular.
\end{proof}

\begin{cor}
Every type ${\rm III_1}$ factor $M$ with separable predual satisfying {\rm CBP} has a singular maximal abelian $\ast$-subalgebra $A \subset M$ with normal expectation.
\end{cor}

\begin{proof}
By \cite[Theorem 3.1]{Ha85}, there exists a faithful state $\varphi \in M_\ast$ such that $(M_\varphi)' \cap M = \C 1$. By Theorem \ref{thm-singular}, there exists an abelian $\ast$-subalgebra $A \subset M_\varphi$ that is maximal abelian and singular in $M$. Moreover, $A \subset M$ is the range of a normal faithful conditional expectation.
\end{proof}


\section{Stability of CBP under finite index extensions/restrictions}\label{section:CBP-stability}

In this section we investigate the stability properties of CBP for inclusions of type ${\rm III_1}$ factors $N \subset M$ with normal faithful conditional expectation $\rE_N : M \to N$. Fix a normal faithful state $\varphi$ on $M$ such that $\varphi = \varphi \circ \rE_N$. The bicentralizer algebras $\rB(N, \varphi)$ and $\rB(M, \varphi)$ are not related in any obvious way and so it is hopeless to try to prove in general that CBP passes to subalgebras or overalgebras. In fact, any type ${\rm III_1}$ factor $N$ embeds in an irreducible way and with NCE into a type ${\rm III_1}$ factor $M$ that satisfies CBP. Indeed, choose any normal faithful state $\psi_N$ on $N$ and put $(M, \psi_M) = (N, \psi_N) \ast (\rL(\Z_2), \tau_{\Z_2})$. It follows from \cite[Theorem A.1]{HU15} that $M$ is a type ${\rm III_1}$ factor that satisfies CBP,  $N \subset M$ is with NCE and $N' \cap M = \C 1$.

However, when the inclusion $N\subset M$ has finite index, we show here that $N$ satisfies CBP if and only if $M$ satisfies CBP

\begin{thm}\label{thm-CBP-finite-index}
Let $N \subset M$ be any inclusion of type ${\rm III_1}$ factors with separable predual such that $N$ is the range of a normal faithful conditional expectation and $N$ has finite index in $M$. 

Then $N$ satisfies {\rm CBP} if and only if $M$ satisfies {\rm CBP}.
\end{thm}

\begin{proof}
{\bf If $N$ satisfies CBP then $M$ satisfies CBP.} Denote by $(M, \rL^2(M), J, \rL^2(M)_+)$ the standard form of $M$. Fix a normal faithful conditional expectation $\rE_N : M \to N$ with finite index. Denote by $\langle M, N \rangle := (JNJ)' \cap \mathbf B(\rL^2(M))$ the Jones basic construction and by $e_N : \rL^2(M) \to \rL^2(N) : x \xi_\varphi \mapsto \rE_N(x)\xi_\varphi$ the Jones projection. We denote by $\Phi : \langle M, N\rangle \to M$ the canonical normal faithful conditional expectation (see \cite{Ko85}).

Since $N$ satisfies CBP, by \cite[Theorem 3.1]{Ha85}, there exists a faithful state $\varphi \in M_\ast$ such that $\varphi = \varphi \circ \rE_N$ and $(N_\varphi)' \cap N = \C 1$. Put $P = (N_\varphi)' \cap M$ and observe that $P \subset M$ is globally invariant under $\sigma^\varphi$. Let $e : \rL^2(P) \to \C \xi_\varphi$ be the rank-one orthogonal projection. Since $(N_\varphi)' \cap N = \C 1$, we have $\rE_N(x) = \varphi(x) 1$ for every $x \in P$. Let $\langle P, e \rangle = (P \cup \{e\})\dpr$ and observe that $\langle P, e \rangle = \mathbf B(\rL^2(P))$. Denote by $\psi_{\langle P, e \rangle}$ (resp.\ $\psi_{\langle M, N \rangle}$) the natural normal faithful semifinite weight defined on $\langle P, e \rangle$ (resp.\ $\langle M, N\rangle$). Then the map $$V : \rL^2(\langle P, e\rangle, \psi_{\langle P, e\rangle}) \to \rL^2(\langle M, N\rangle, \psi_{\langle M, N\rangle}) : \Lambda_{\psi_{\langle P, e\rangle}}(xey) \mapsto \Lambda_{\psi_{\langle M, N\rangle}}(xe_Ny), \quad x, y \in P$$ is an isometry. Denote by $\mathcal P \subset \langle M, N\rangle$ the weak closure of the (possibly) nonunital $\ast$-subalgebra $\spn \{x e_N y : x, y \in P\}$. Observe that $\sigma_t^{\psi_{\langle M, N\rangle}}(x \, e_N \, y) = \sigma_t^\varphi(x) \, e_N \, \sigma_t^\varphi(y)$ for all $t \in \R$ and all $x, y \in P$. It follows that $\spn \{x e_N y : x, y \in P\}$ is globally invariant under $\sigma^{\psi_{\langle M, N\rangle}}$. Thus, we have $\sigma_t^{\psi_{\langle M, N\rangle}}(1_{\mathcal P}) = 1_{\mathcal P}$ for every $t \in \R$ and $\mathcal P \subset 1_{\mathcal P} \langle M, N\rangle 1_{\mathcal P}$ is globally invariant under $\sigma^{\psi_{\langle M, N\rangle}}$. Observe that $e_N \leq 1_{\mathcal P}$. Since $\psi_{\langle M, N\rangle}(1_{\mathcal P}  \cdot  1_{\mathcal P})$ is semifinite on $\mathcal P$, there exists a $\psi_{\langle M, N\rangle}(1_{\mathcal P} \cdot 1_{\mathcal P})$-preserving conditional expectation $\rE_{\mathcal P} : 1_{\mathcal P}\langle M, N\rangle 1_{\mathcal P} \to \mathcal P$ (see \cite[Theorem IX.4.2]{Ta03}). Then the projection $VV^*$ is nothing but the orthogonal projection $\rL^2(\langle M, N\rangle, \psi_{\langle M, N\rangle}) \to \rL^2(\mathcal P, \psi_{\langle M, N\rangle}(1_{\mathcal P} \cdot 1_{\mathcal P}))$. Write $\rL^2(\mathcal P) = \rL^2(\mathcal P, \psi_{\langle M, N\rangle}(1_{\mathcal P} \cdot 1_{\mathcal P}))$. Thus, the map $\Theta : \langle P, e\rangle \to \mathbf B(\rL^2(\mathcal P))$ defined by $\Theta(a) V = V a$ for $a \in \langle P, e \rangle$ is a unital normal $\ast$-embedding that satisfies $\Theta(xey) = xe_N y$ for all $x, y \in P$. In particular, $\Theta(\langle P, e \rangle) \subset \mathcal P$ and $\Theta : \langle P, e\rangle \to \mathcal P$ is a $P$-$P$-bimodular normal faithful unital completely positive map.

Regard $\langle P, e\rangle = \mathbf B(\rL^2(P))$ and define $\Psi : \mathbf B(\rL^2(P)) \to P$ by the composition $\Psi = \rE_{P} \circ \Phi \circ \Theta$. Then $\Psi$ is a $P$-$P$-bimodular normal faithful completely positive map. Since $\Psi(1) \in \mathcal Z(P)_+$, we may choose a nonzero element $c \in \mathcal Z(P)_+$ so that $z := c^{1/2}\Psi(1)c^{1/2}$ is a nonzero projection in $\mathcal Z(P)$. Then the map $\Psi_z : \mathbf B(\rL^2(Pz)) \to Pz$ defined by $\Psi_z := c^{1/2} \Psi (z \, \cdot \, z) c^{1/2}$ is a $Pz$-$Pz$-bimodular normal unital completely positive map and hence a normal conditional expectation. Therefore $Pz$ is a discrete von Neumann algebra. By \cite[Theorem 3.5]{HI15}, the bicentralizer algebra $\rB(M, \varphi)$ satisfies the following dichotomy: either $\rB(M, \varphi) = \C1$ or $\rB(M, \varphi)$ is a type ${\rm III_1}$ factor. Since the inclusions $\rB(M, \varphi) \subset (M_\varphi)' \cap M \subset (N_\varphi)' \cap M=  P$ are all globally invariant under $\sigma^\varphi$ and since $P$ has a minimal projection, it follows that $\rB(M, \varphi) = \C 1$. Therefore, $M$ satisfies CBP.

{\bf If $M$ satisfies CBP then $N$ satisfies CBP.} Since the inclusion $N \subset M$ has finite index, we may choose a normal faithful conditional expectation $\rE_N : M \to N$ for which there exists $\kappa > 0$ such that $\rE_N(x) \geq \kappa \, x$ for every $x \in M_+$ (see \cite{PP84, Po95}). Let $\varphi$ be a normal faithful state on $M$ such that $\varphi = \varphi \circ \rE_N$. Fix a nonprincipal ultrafilter $\omega \in \beta(\N) \setminus \N$ and denote by $M^\omega$ (resp.\ $N^\omega$) the Ocneanu ultraproduct of $M$ (resp.\ $N$) with respect to $\omega$ (see \cite{Oc85, AH12}). Following \cite[Section 1.3]{Po95}, define the normal faithful conditional expectation $\rE_{N^\omega} : M^\omega \to N^\omega$ by the formula $\rE_{N^\omega}((x_n)^\omega) = (\rE_N(x_n))^\omega$ for every $(x_n)^\omega \in M^\omega$. Then we have $\rE_{N^\omega}(x) \geq \kappa \, x$ for every $x \in (M^\omega)_+$. Put $\mathcal N = (N^\omega)_{\varphi^\omega}$ and $\mathcal M = (M^\omega)_{\varphi^\omega}$ and observe that both $\mathcal N$ and $\mathcal M$ are type ${\rm II_1}$ factors by \cite[Proposition 4.24]{AH12}. Since $\varphi^\omega = \varphi^\omega \circ \rE_{N^\omega}$ and since $\rE_{N^\omega}(x) \in \mathcal N$ for every $x \in \mathcal M$, we have $\rE_{\mathcal N}(x) = \rE_{N^\omega}(x) \geq \kappa \, x$ for every $x \in \mathcal M_+$, where $\rE_{\mathcal N} : \mathcal M \to \mathcal N$ denotes the unique trace preserving conditional expectation. Thus, the inclusion $\mathcal N \subset \mathcal M$ has finite index by \cite[Theorem 2.2]{PP84}.

We first prove that $\mathcal M' \cap M^\omega = \C 1$. Since $M$ satisfies CBP, by \cite[Theorem 3.1]{Ha85}, there exists a normal faithful state $\psi$ on $M$ such that $(M_\psi)' \cap M = \C 1$. Then \cite[Lemma 2.3]{Po81} implies that $((M_\psi)^\omega)' \cap M^\omega = \C 1$. Since $(M_\psi)^\omega \subset (M^\omega)_{\psi^\omega}$, we have $((M^\omega)_{\psi^\omega})' \cap M^\omega = \C 1$. By Connes--St\o rmer transitivity theorem \cite{CS76} (see also \cite[Theorem 4.20]{AH12}), there exists $u \in \mathcal U(M^\omega)$ such that $\psi^\omega = u \varphi^\omega u^*$. Thus, we obtain $$\mathcal M' \cap M^\omega = ((M^\omega)_{\varphi^\omega})' \cap M^\omega = u^* (((M^\omega)_{\psi^\omega})' \cap M^\omega) u = \C 1.$$

We next prove that $\mathcal N' \cap M^\omega$ has a nonzero minimal projection following the lines of \cite[Lemma 3.3]{Po09}.  Since $\mathcal N \subset \mathcal M$ is a finite index inclusion of type ${\rm II_1}$ factors, we may choose a projection $e \in \mathcal M$ such that $\rE_{\mathcal N}(e) = [\mathcal M : \mathcal N]^{-1} 1$. Put $\mathcal P = \{e\}' \cap \mathcal N$ so that $\mathcal P \subset \mathcal N$ is a finite index inclusion of type ${\rm II_1}$ factors and $\mathcal M = \langle \mathcal N, e\rangle$ (see \cite[Corollary 1.8]{PP84}). By \cite[Proposition 1.3]{PP84}, choose a finite basis $(X_j)_{j \in J}$ of $\mathcal N$ over $\mathcal P$. Recall that we have $\sum_{j \in J} X_j e X_j^* = 1$, $\sum_{j \in J} X_j X_j^* = [\mathcal M : \mathcal N]$ and for every $j \in J$, $p_j := \rE_{\mathcal P}(X_j^*X_j)$ is a projection in $\mathcal P$. Since $\sum_{j \in J}  X_j e \, x \, e X_j^* \in \mathcal M' \cap M^\omega = \C 1$ for every $x \in \mathcal N' \cap M^\omega$, we may define the state $\Psi \in (\mathcal N' \cap M^\omega)_\ast$ by the formula $\sum_{j \in J} X_j e \, x \, e X_j^* = \Psi(x) 1$. Moreover, we have $exe = \Psi(x) e$ for every $x \in \mathcal N' \cap M^\omega$. Following \cite[Lemma 3.3]{Po09}, put $b = [\mathcal M : \mathcal N] \cdot \rE_{\mathcal N' \cap \mathcal M}(e) = [\mathcal M : \mathcal N] \cdot \rE^{\varphi^\omega}_{\mathcal N' \cap M^\omega}(e) \in (\mathcal N' \cap \mathcal M)_+$. For every $x \in \mathcal N' \cap M^\omega$, we have
\begin{align*}
\varphi^\omega \left(\sum_{j \in J} X_j e b^{-1/2} \, x \, b^{-1/2} e X_j^*\right) &= \sum_{j \in J} \varphi^\omega ( x \, b^{-1/2} e X_j^*  X_j e b^{-1/2}) = \sum_{j \in J} \varphi^\omega ( x \, b^{-1/2} p_j e b^{-1/2}) \\ 
&= \sum_{j \in J} \varphi^\omega \circ \rE^{\varphi^\omega}_{\mathcal P' \cap M^\omega}( x \, b^{-1/2} p_j e b^{-1/2}) \\
&= [\mathcal M : \mathcal N] \cdot \varphi^\omega ( x \, b^{-1/2} e b^{-1/2}) = \varphi^\omega(x).
\end{align*}
Since $\sum_{j \in J} X_j e b^{-1/2} \, x \, b^{-1/2} e X_j^* \in \mathcal M' \cap M^\omega = \C 1$ for every $x \in \mathcal N' \cap M^\omega$, we have
$$\forall x \in \mathcal N' \cap M^\omega, \quad \sum_{j \in J} X_j e b^{-1/2} \, x \, b^{-1/2} e X_j^* = \varphi^\omega(x) 1.$$ 
Moreover, $f = b^{-1/2} e b^{-1/2} \in \mathcal M$ is a projection such that $f x f = \varphi^\omega(x) f$ for every $x \in \mathcal N' \cap M^\omega$. This implies that the von Neumann algebra $\langle \mathcal N' \cap M^\omega , f\rangle = ((\mathcal N' \cap M^\omega) \cup \{f\})\dpr$ has a minimal projection, namely $f$. Since $f \in \mathcal M$, the von Neumann subalgebra $\langle \mathcal N' \cap M^\omega , f\rangle \subset M^\omega$ is globally invariant under $\sigma^{\varphi^\omega}$. Since the inclusions $\mathcal N' \cap M^\omega \subset \langle \mathcal N' \cap M^\omega , f\rangle \subset M^\omega$ are all globally invariant under $\sigma^{\varphi^\omega}$, we obtain that $\mathcal N' \cap M^\omega$ has a minimal projection as well.

By \cite[Proposition 3.3]{HI15}, we have $\rB(N, \varphi) = ((N^\omega)_{\varphi^\omega})' \cap N = \mathcal N' \cap N$. Since the inclusion $\rB(N, \varphi) = \mathcal N' \cap N \subset \mathcal N' \cap M^\omega$ is globally invariant under $\sigma^{\varphi^\omega}$ and since $\mathcal N' \cap M^\omega$ has a minimal projection, \cite[Theorem 3.5]{HI15} implies that $\rB(N, \varphi) = \C 1$. Therefore, $N$ satisfies CBP.
\end{proof}

\section{Open problems}

Formulated some forty years ago and still open, Connes's Bicentralizer Problem remains one of the most famous unsolved problems in von Neumann algebras. 
It is certainly the central, most important open problem in the theory of type ${\rm III_1}$ factors. 
The fundamental role it plays in unraveling the structure of  type ${\rm III_1}$ factors 
comes from its equivalent form as existence of a normal faithful state with large centralizer  (due to \cite{Ha85}). 
In turn, this latter form of CBP (often accompanied by Connes--Stormer's theorem) allows 
adapting  arguments from II$_1$ factors to the ``III$_1$ factor world''. For instance, 
it has been a key feature in developing a type ${\rm III_1}$ version of the second named author's deformation-rigidity theory, 
which has been initially developed in II$_1$ factor framework (see e.g.\ \cite{HI14, HI15}).

It is somewhat notorious that Connes and Haagerup strongly believed CBP had an affirmative answer. 
But since all efforts to prove it have failed, during the last decade 
there have been attempts to  produce counterexamples as well, in fact  
some of the papers involving the first named author have been motivated 
by such attempts (see e.g.\ \cite{Ho08}). 

However, at this moment, both authors of this paper believe CBP has a positive answer. The purpose of the previous section 
was to offer some supporting evidence in this respect, with its 
partial results bound to become redundant if CBP is proven in its full generality. We'll formulate in this section  
several related problems, 
including some  stronger versions of the CBP conjecture. 

Let us first recall that Connes' Bicentralizer Property for a type ${\rm III_1}$ factor with separable predual $M$ 
was shown in \cite{Ha85} to be equivalent to the weak relative Dixmier property 
of the inclusion $M_\Phi \subset M$, where $\Phi$ is any normal faithful dominant weight on $M$ and $M_\Phi$ denotes the 
fixed point algebra of its automorphism group. 

The terminology  {\it weak relative Dixmier property} for an inclusion of (arbitrary) von Neumann algebras $N\subset M$ 
is in the sense of \cite{Po98},  and it means  
that the convex set $\mathcal K_N(x)=\overline{\text{\rm co}}^w \{uxu^* \mid u\in \mathcal U(N)\}$ has non-empty intersection 
with $N'\cap M$, for any $x\in M$. Note that if $M=\mathbf B(H)$ then this condition for $N\subset M$ is equivalent to 
$N$ being amenable (cf.\ \cite{Sc63}). Note also that in the 
case $N\subset M$ is an irreducible inclusion of factors (i.e.,  if $N'\cap M=\mathbf C1$), 
then for this condition to hold true it is sufficient to have an 
amenable (equivalently approximately finite dimensional, by \cite{Co75}) 
von Neumann subalgebra $B\subset N$ such that $B'\cap M\subset B$ (thus $B'\cap M = \mathcal Z(B)$). 
Indeed, because then by \cite{Sc63} we have $\mathcal K_B(x)\cap B'\cap B\neq \emptyset$ for all 
$x\in M$, and by applying in the factor $N$ the Dixmier averaging theorem \cite{Di57} to an element in this intersection set, one gets $\mathcal K_N(x)\cap \mathbf C1\neq \emptyset$ (see \cite[Remark 3.9]{Ha85}). 

In particular, 
if  $N\subset M$ is an irreducible inclusion of factors that satisfies {\it Kadison's property}, i.e., 
$N$ contains an abelian von Neumann subalgebra that's a MASA in $M$, then $N\subset M$ has the weak relative Dixmier property. 
The existence of a MASA in a subfactor $N\subset M$ clearly implies irreducibility, 
and one of the well known problems  in \cite{Ka67} asks whether the converse is true as well, i.e., if 
Kadison's property actually characterizes irreducibility. 

It is easy to see that if $N\subset M$ is an inclusion of von Neumann algebras with NCE 
and $N$ is semifinite, then $N\subset M$ satisfies the weak relative Dixmier property (see \cite{Po81, Po98}). 
It has in fact been shown in \cite{Po81} that if in addition $N'\cap M\subset N$, then  $N$  contains a MASA of $M$ with NCE 
(so such $N\subset M$ do satisfy Kadison's property), and that if $N\subset M$ is irreducible then $N$ 
contains a hyperfinite subfactor $R \subset N$ with NCE such that $R'\cap M=\mathbf C 1$.

\begin{prob}
Let $N \subset M$ be an irreducible inclusion with NCE of type ${\rm III_1}$ factors with 
separable predual such that  $N$ satisfies CBP.  Does $N$ 
contain an  amenable (or even abelian) von Neumann subalgebra $B\subset N$ with NCE and such  that $B'\cap M\subset B$? Is this at least true if $N$ is the hyperfinite type ${\rm III_1}$ factor? 
\end{prob}

Note that by a result in \cite{GP96}, 
there do exist examples of irreducible inclusions of factors $N\subset M$ with $N$ of type II$_1$, 
$M$ of type II$_\infty$ such that $N$ contains no amenable 
von Neumann subalgebra $B$ with the property 
that $B'\cap M\subset B$. But the examples of irreducible inclusions in \cite{GP96} that do not satisfy Kadison's 
property are not with NCE. Thus,  the problem of whether Kadison's criterion characterizes irreducibility 
for an inclusion of factors seems quite subtle, in its full generality. 

In turn, the weak relative Dixmier property may still be true for arbitrary irreducible inclusions.

\begin{prob}
Let $N \subset M$ be an arbitrary irreducible 
inclusion of factors with separable predual. Does $N\subset M$ have the weak 
relative Dixmier property? Is this at least true when the NCE condition is satisfied? 
\end{prob}

As we mentioned before, if this is true in the case $M$ is type ${\rm III_1}$ and $N$ is its type  ${\rm II_\infty}$ core then CBP holds true. Note that by \cite{Po81}, if $N$ is any non-Gamma type ${\rm II_1}$ factor (e.g., if $N$ is the free group factor $\rL(\mathbf F_n)$, cf \cite{MvN43}) and $M$ 
is the ultrapower factor $N^\omega$, for some nonprincipal ultrafilter on $\mathbf N$, then $N\subset M$ is irreducible, yet $N$ contains no MASAs of $M$. 
But in these examples the larger factor is non-separable. However, such inclusions $N\subset M$ do satisfy the weak relative Dixmier property.


\bibliographystyle{plain}

\end{document}